\documentclass[10pt]{article}
\usepackage{a4wide, inputenc, amsmath, amsthm, amssymb, dsfont, setspace, stmaryrd, verbatim, graphicx, latexsym, enumerate, tipa, caption, mathrsfs, mathabx, pdfpages, blkarray, tikz, pst-node, pifont}
\usepackage[all,cmtip]{xy}
\usepackage{colortbl}

\newcommand{\lr}[1]{\langle #1 \rangle}

\newcommand{\dom}{\text{dom}}
\newcommand{\cmark}{\ding{51}}
\newcommand{\xmark}{\ding{55}}

\theoremstyle{definition}

\newtheorem{theorem}{Theorem}

\newtheorem*{claim}{Claim}
\newtheorem{definition}[theorem]{Definition}

\newtheorem{lemma}[theorem]{Lemma}
\newtheorem{fact}[theorem]{Fact}
\newtheorem{example}[theorem]{Example}
\newtheorem{question}[theorem]{Question}
\newtheorem{corollary}[theorem]{Corollary}

\newcounter{saveenumi}

\title{Proper and Improper Variants of Mathias and Silver Forcing}
\author{Liu Shixiao}
\date{}

\begin{document}
\begin{spacing}{1.2}
    
\maketitle

\begin{abstract}
In this paper, we answer several questions in \cite{LMS-R} regarding density variants of Mathias and Silver forcing. These questions include whether each of the forcings is proper, adds dominating reals, or adds Cohen reals. We also generalize one of the proofs to Mathias forcings parametrized by lower semi-continuous submeasures satisfying certain properties.
\end{abstract}
\section{Introduction}

The aim of this paper is to establish several new results on variants of Mathias and Silver forcing introduced by Laguzzi, Mildenberger, and Stuber-Rousselle in \cite{LMS-R}.
These variants of Mathias and Silver forcing differ from previously considered ones as they are parametrized neither by filters nor by co-ideals, but by families of sets that satisfy certain conditions on their upper or lower density.

One of the reasons for analyzing such variants of Mathias and Silver forcing is to gain a better understanding of cardinals $\text{cov}^*(\mathcal{Z}_{0})$ and $\text{non}^*({\mathcal{Z}}_{0})$, which are cardinal characteristics associated to the ideal ${\mathcal{Z}}_{0}$ of sets of asymptotic density $0$.
Recall that:
\begin{definition}
Let $A \subseteq \omega$. The upper density and lower density of $A$ are defined by
 \begin{align*}
  &{d}^{-}(A) = \displaystyle\liminf_{n \rightarrow \infty}{\frac{\vert A \cap n \vert}{n}}.\\
  &{d}^{+}(A) = \displaystyle\limsup_{n \rightarrow \infty}{\frac{\vert A \cap n \vert}{n}}.
 \end{align*}
In the case of ${d}^{-}(A) = {d}^{+}(A)=x$, we also say the density of $A$ is $x$ and write $d(A) = x$. And we define
\[\mathcal{Z}_0 = \{ A \subseteq \omega: d(A) = 0 \}\]
\end{definition}

It is well-known (see \cite{ilijasbook} or \cite{MR1708146}) that ${\mathcal{Z}}_{0}$ is an analytic P-ideal. Also recall that:

\begin{definition}
A function $\varphi: \mathcal{P}(\omega) \rightarrow [0, \infty]$ is called a \emph{submeasure on $\omega$} if
\begin{enumerate}
\item $\varphi(\emptyset) = 0$;
\item $\varphi(X) \leq \varphi(Y)$, for all $X \subseteq Y \subseteq \omega$;
\item $\varphi(X \cup Y) \leq \varphi(X) + \varphi(Y)$, for all $X, Y \in \mathcal{P}(\omega)$;
\item $\varphi(\{n\}) < \infty$, for all $n \in \omega$.
\end{enumerate}

We say a submeasure $\varphi$ on $\omega$ is \emph{lower semi-continuous}, if $\varphi(X) = \lim\limits_{n \rightarrow \infty}{\varphi(X \cap n)}$ for every $X \in \mathcal{P}(\omega)$.
For a lower semi-continuous submeasure $\varphi$ on $\omega$, we define
\[\text{Exh}(\varphi) = \{ X \subseteq \omega: \lim\limits_{m \rightarrow \infty}{\varphi(X \setminus m)} = 0 \}\]
\end{definition}
It is not hard to see that $\text{Exh}(\varphi)$ is an ${F}_{\sigma\delta}$ P-ideal for every lower semi-continuous submeasure $\varphi$ on $\omega$. Indeed, analytic P-ideals are always of this form; by a theorem of Solecki~\cite{MR1708146}, every analytic P-ideal $\mathcal{I}$ is ${F}_{\sigma\delta}$ and there is a lower semi-continuous submeasure $\varphi$ on $\omega$ such that $\mathcal{I} = \text{Exh}(\varphi)$. Cardinal invariants associated with such analytic ideals as well as properties of their quotients have been well-studied. For example, in \cite{ilijasbook}, Farah studied gaps in quotients of the form $\mathcal{P}(\omega) \slash \mathcal{I}$, where $\mathcal{I}$ is an analytic ideal, and in \cite{MR2254542} he proved that $\mathcal{P}(\omega)~\slash~{\mathcal{Z}}_{0}$ is forcing equivalent to $( \mathcal{P}(\omega) \slash \text{FIN} ) * \mathcal{B}$, where $\mathcal{B}$ is a specific measure algebra.
Brendle and Shelah~\cite{MR1686797} defined and investigated certain cardinal invariants for analytic ideals.
In \cite{MR2319159}, Hern\'{a}ndez-Hern\'{a}ndez and Hru\v{s}\'{a}k focused on four cardinal invariants associated to any tall analytic P-ideal, establishing various connections with the classical cardinal invariants.
\begin{definition}
A non-principal ideal $\mathcal{I}$ on $\omega$ is said to be \emph{tall} if for all $A \in [\omega]^{{\aleph}_{0}}$ there is $B \in [A]^{{\aleph}_{0}}$ such that $B \in \mathcal{I}$.
For a non-principal tall ideal $\mathcal{I}$ on $\omega$, define the cardinals
\begin{align*}
  {\text{add}}^{*}(\mathcal{I}) &= \min\{ \vert \mathcal{F} \vert: \mathcal{F} \subseteq \mathcal{I} \wedge \forall X \in \mathcal{I} \,\exists A \in \mathcal{F}\,(A \, {\not\subseteq}^{*} \, X )\}\\
  {\text{cov}}^{*}(\mathcal{I}) &= \min\{ \vert \mathcal{F} \vert: \mathcal{F} \subseteq \mathcal{I} \wedge \forall X \in [\omega]^{{\aleph}_{0}} \exists A \in \mathcal{F}\,( \vert A \cap X \vert = {\aleph}_{0} )\}\\
  {\text{cof}}^{*}(\mathcal{I}) &= \min\{ \vert \mathcal{F} \vert: \mathcal{F} \subseteq \mathcal{I} \wedge \forall X \in \mathcal{I} \,\exists A \in \mathcal{F}\,(X \,{\subseteq}^{*} \, A )\}\\
  {\text{non}}^{*}(\mathcal{I}) &= \min\{ \vert \mathcal{F} \vert: \mathcal{F} \subseteq [\omega]^{{\aleph}_{0}} \wedge \forall X \in \mathcal{I}\, \exists A \in \mathcal{F}\,( \vert A \cap X \vert < {\aleph}_{0} )\}.
\end{align*}   
\end{definition}
The cardinals ${\text{cov}}^{*}(\mathcal{I})$ and ${\text{non}}^{*}(\mathcal{I})$ are duals.
Similarly, ${\text{add}}^{*}(\mathcal{I})$ and ${\text{cof}}^{*}(\mathcal{I})$ are duals. See \cite{MR2768685} for more about duality of cardinal invariants. For definitions of the classical cardinal invariants, such as $\mathfrak{b},\mathfrak{d},\mathfrak{s},\mathfrak{r}$, the reader can refer to Blass~\cite{MR2768685} or to Bartoszy\'{n}ski and Judah~\cite{MR1350295}. In \cite{MR2319159}, Hern\'{a}ndez-Hern\'{a}ndez and Hru\v{s}\'{a}k pointed {out} that ${\text{cov}}^{*}({\mathcal{Z}}_{0})$ and ${\text{non}}^{*}({\mathcal{Z}}_{0})$ play a basic role in the general theory these invariants.
They established (see Theorems 3.10 and 3.12 of \cite{MR2319159}) that $\min\{\text{cov}(\mathcal{N}), \mathfrak{b}\} \leq {\text{cov}}^{*}({\mathcal{Z}}_{0}) \leq \text{non}(\mathcal{M})$ , $\text{cov}(\mathcal{M}) \leq {\text{non}}^{*}({\mathcal{Z}}_{0}) \leq \max\{\mathfrak{d}, \text{non}(\mathcal{N})\}$, ${\text{cov}}^{*}({\mathcal{Z}}_{0}) \leq \max\{\mathfrak{b}, \text{non}(\mathcal{N})\}$, and that ${\text{non}}^{*}({\mathcal{Z}}_{0}) \leq \min\{\mathfrak{d}, \text{cov}(\mathcal{N})\}$.
They raised the question of whether ${\text{cov}}^{*}({\mathcal{Z}}_{0}) \leq \mathfrak{d}$.
Raghavan and Shelah~\cite{MR3615051} answered their question by proving the dual inequalities ${\text{cov}}^{*}({\mathcal{Z}}_{0}) \leq \mathfrak{d}$ and $\mathfrak{b} \leq {\text{non}}^{*}({\mathcal{Z}}_{0})$ in ZFC.
Raghavan~\cite{MR4099835} further improved these results by showing in ZFC that $\min\{\mathfrak{d}, \mathfrak{r}\} \leq {\text{non}}^{*}({\mathcal{Z}}_{0})$ and that ${\text{cov}}^{*}({\mathcal{Z}}_{0}) \leq \max\{\mathfrak{b}, \mathfrak{s}(\mathfrak{pr})\}$.
Here, $\mathfrak{s}(\mathfrak{pr})$ is defined to be the minimal cardinality of a family $\mathcal{F}$ of partitions of $\omega$ into infinitely many pieces such that for every $X \in [\omega]^{{\aleph}_{0}}$, there is a partition in $\mathcal{F}$ such that all pieces have infinite intersection with $X$.
It is not known whether $\mathfrak{s}(\mathfrak{pr}) = \mathfrak{s}$.

The cardinal ${\text{cov}}^{*}(\mathcal{I})$ is closely related to the question of which forcing notions can diagonalize $\mathcal{I}$.
For any ideal $\mathcal{I} \in {V}$, a forcing $\mathbb{P} \in {V}$ is said to \emph{diagonalize} $\mathcal{I}$ if there is a $\mathbb{P}$-name $\dot{X}$ such that ${\Vdash}_{\mathbb{P}}{\dot{X} \in [\omega]^{{\aleph}_{0}}}$ and for each $I \in {V} \cap \mathcal{I}$, ${\Vdash}_{\mathbb{P}}{\vert \dot{X} \cap I \vert < {\aleph}_{0}}$.
Note that if $\mathbb{P}$ diagonalizes $\mathcal{I}$, then $\mathbb{P}$ tends to increase ${\text{cov}}^{*}(\mathcal{I})$.
Furthermore, if $\mathcal{I}$ is \emph{Kat{\v e}tov below} $\mathcal{J}$, which means that $\exists f \in {\omega}^{\omega} \forall I \in \mathcal{I}\,({f}^{-1}(I) \in \mathcal{J})$, then any $\mathbb{P}$ that diagonalizes $\mathcal{J}$ will also diagonalize $\mathcal{I}$.
The question of whether an ideal can be diagonalized without adding certain types of reals has a long history in set theory of the reals.
For instance, Laflamme \cite{MR1068126} showed that every ${F}_{\sigma}$ ideal can be diagonalized by an ${\omega}^{\omega}$-bounding forcing, and Canjar \cite{MR0969054} showed that certain ultrafilters can be diagonalized without adding dominating reals.
The result of \cite{MR3615051} shows that any proper forcing which diagonalizes ${\mathcal{Z}}_{0}$ must add an unbounded real, and it is proved in \cite{MR4099835} that any Suslin c.c.c.\@ forcing which diagonalizes ${\mathcal{Z}}_{0}$ must add a dominating real.
Raghavan \cite{MR4099835} also asked whether ${\mathcal{Z}}_{0}$ can be diagonalized by any proper forcing without adding dominating reals; this question still remains open.

Brendle, Guzm\'{a}n, Hru\v{s}\'{a}k and Raghavan \cite{2022arXiv220614936B} have recently shown that the ideal generated by a m.a.d.\@ family is always \emph{nearly} Kat{\v e}tov below ${\mathcal{Z}}_{0}$, which means that the reduction holds after removing countably many members of the m.a.d.\@ family. It is a long-standing open problem whether $\mathfrak{d} = {\aleph}_{1}$ implies $\mathfrak{a} = {\aleph}_{1}$.
In \cite{MR3129732}, Brendle and Raghavan raised the question of whether $\mathfrak{b} = \mathfrak{s} = {\aleph}_{1}$ already implies $\mathfrak{a} = {\aleph}_{1}$, and their question remains open as well.
The results of \cite{2022arXiv220614936B} show that an understanding of what forcing notions are capable of diagonalizing ${\mathcal{Z}}_{0}$ is useful for analyzing $\mathfrak{a}$.
By \cite{MR3129732}, for any collection of closed sets whose union is a m.a.d.\@ family, there is a forcing that does not add dominating reals and forces that the reinterpreted union is not maximal in the extension.

Laguzzi, Mildenberger, and Stuber-Rousselle~\cite{LMS-R} recently introduced new variants of Mathias and Silver forcing parametrized by upper or lower density.
Their variants of Mathias forcing diagonalize ${\mathcal{Z}}_{0}$.
\begin{definition}
The Mathias forcing $\mathbb{M}$ consists of tuples $(s,A)$, where {$s\subset \mathbb{N}$ is finite, $A\subseteq \mathbb{N}$ is infinite}, and $\max s <\min A$. For two tuples, $(s_1,A_1)\leq (s_2,A_2)$ iff
\begin{itemize}
    \item $s_1\supseteq s_2$
    \item $A_1\subseteq A_2$
    \item If $n\in s_1\setminus s_2$, then $n\in A_2$
\end{itemize}
The density variants of Mathias forcing are denoted with superscripts $^-$ $^+$ and subscripts $_{\geq \varepsilon}$ $_{> 0}$, where the density requirements are imposed on the infinite set $A$. For example, $\mathbb{M}^-_{\geq \varepsilon}$ requires $d^-(A)\geq \varepsilon$ for all conditions $(s,A)$, and $\mathbb{M}^+_{>0}$ requires $d^+(A)>0$.
\end{definition}
\begin{definition}
The Silver forcing $\mathbb{V}$ consists of partial functions $f:\omega\rightarrow 2$ such that $\omega\setminus\dom(f)$ is infinite. For two conditions, $f_1\leq f_2$ iff $f_1\supseteq f_2$. The density variants of Silver forcing are denoted with superscripts $^-$ $^+$ and subscripts $_{\geq \varepsilon}$ $_{> 0}$, where the density requirements are imposed on $\omega\setminus\dom(f)$. For instance, $\mathbb{V}^-_{\geq \varepsilon}$ requires $d^-(\omega\setminus\dom(f))\geq \varepsilon$ for all conditions $f$, and $\mathbb{V}^+_{>0}$ requires $d^+(\omega\setminus\dom(f))>0$.
\end{definition}
Properties of the classical Mathias and Silver models are well-known. See \cite{MR1350295}.
Nevertheless new properties of these models are still being discovered, for instance \cite{MR3990958} showed there are no P-points in the Silver model.  While many variants of Mathias and Silver forcing have been considered previously, for example \cite{MR0297560} and \cite{MR3142391}, the ones considered by Laguzzi, Mildenberger, and Stuber-Rousselle in \cite{LMS-R} are quite different because they are determined by neither a filter nor a co-ideal. Their results are summarized in the table below. {Note that the paper also claims to have proven the properness of $\mathbb{M}_{\geq\epsilon}^{-}$, but we will soon show an error in their proof.}

\begin{center}
\begin{tabular}{|c|c|c|c|c|c|c|c|c|}
\hline
~ & $\mathbb{M}^-_{>0}$ & $\mathbb{M}^+_{>0}$ & $\mathbb{M}^-_{\geq\epsilon}$ & $\mathbb{M}^+_{\geq\epsilon}$ & $\mathbb{V}^-_{>0}$ & $\mathbb{V}^+_{>0}$ & $\mathbb{V}^-_{\geq\epsilon}$ & $\mathbb{V}^+_{\geq\epsilon}$ \\
\hline
proper & ~ & \cmark & ~ & \cmark & \xmark & ~ & ~ & \cmark\\
\hline
dominating & ~ & \cmark & ~ & \cmark & \cmark & ~ & ~ & \xmark\\
\hline
Cohen & \cmark & \cmark & \cmark & \cmark & \cmark & ~ & ~ & \xmark\\
\hline
\end{tabular}
\end{center}

Several questions were left open in \cite{LMS-R}, indicated by the blank boxes in the table above. In section $2$ of this paper, we study the Mathias forcings ${\mathbb{M}}^{-}_{\geq \varepsilon}$ and ${\mathbb{M}}^{-}_{>0}$. In particular, Question 3.13 of \cite{LMS-R} asks whether ${\mathbb{M}}^{-}_{\geq\varepsilon}$ adds dominating reals.
They pointed out that since ${\mathbb{M}}^{-}_{\geq\varepsilon}$ diagonalizes ${\mathcal{Z}}_{0}$, a negative answer to their question would give a positive answer to Question 38 of \cite{MR4099835}.
While a general criterion for whether a Mathias-Prikry type forcing adds a dominating real is given in \cite{MR3142391}, this is not applicable to ${\mathbb{M}}^{-}_{\geq\varepsilon}$ because it is not parametrized by a filter.
We will give a positive answer to Question 3.13 of \cite{LMS-R} by showing that ${\mathbb{M}}^{-}_{\geq\varepsilon}$ always adds a dominating real. In addition to this, we will introduce a generalization of ${\mathbb{M}}^{-}_{\geq \varepsilon}$ for a sequence of lower semi-continuous submeasures. In section $3$ we study the Silver forcings ${\mathbb{V}}^{-}_{\geq \varepsilon}$ and ${\mathbb{V}}^{+}_{>0}$. We will show that they both collapse {the} continuum to $\omega$. With all these results, we complete the above table by the end of this paper. Some of the results will rely on the Darboux property of upper and lower density, which can be found in \cite{MR3597402}.

\section{Mathias Forcings with Density}
In this section we deal with properties of two Mathias forcings that remain unproven in \cite{LMS-R}. Let's first show a counterexample to Corollary 3.5 of \cite{LMS-R}. {The corollary claims that, for any set $A$ with $d^-(A)=\epsilon$, the set $\mathcal{F}(A)=\{B\subseteq A:\, d^-(B)=\epsilon\}$ is a filter.} This erroneous corollary was used in the paper to prove the properness of $\mathbb{M}_{\geq\varepsilon}^-$.

\begin{example}
Divide $\mathbb{N}$ into intervals $[a_i,a_{i+1})$ such that $a_0=0$ and $a_{i+1}=2^{a_i}$. We take $A, B_1, B_2$ such that
\[A\cap [a_i,a_{i+1})=
\begin{cases}
    [a_i,a_{i+1})&\text{if $i$ is even}\\
    \{2n:n\in\mathbb{N}\}\cap [a_i,a_{i+1})&\text{if $i$ is odd}
\end{cases}\]
\[B_1\cap [a_i,a_{i+1})=
\begin{cases}
    \{2n:n\in\mathbb{N}\}\cap[a_i,a_{i+1})&\text{if $i$ is even}\\
    \{2n:n\in\mathbb{N}\}\cap [a_i,a_{i+1})&\text{if $i$ is odd}
\end{cases}\]
\[B_2\cap [a_i,a_{i+1})=
\begin{cases}
    \{2n+1:n\in\mathbb{N}\}\cap[a_i,a_{i+1})&\text{if $i$ is even}\\
    \{2n:n\in\mathbb{N}\}\cap [a_i,a_{i+1})&\text{if $i$ is odd}
\end{cases}\]
Then we have $d^-(A)={d^-(B_1)}=d^-(B_2)=1/2$ and $B_1,B_2\subseteq A$, but $d^-(B_1\cap B_2)=0$. Hence $\{B\subseteq A:\, d^-(B)=1/2\}$ is not a filter.

Furthermore, consider the equivalence relation $=^*$ on $2^\omega$, namely $f=^* g$ iff they are equal modulo a finite set. Pick one representative $f_\alpha$ from each equivalence class. The set $\{f_\alpha: \alpha\in I\}$ has size {$2^{\aleph_0}$}. For each $f_\alpha$, we take $B_{\alpha}\subseteq \mathbb{N}$ such that
\[B_{\alpha}\cap [a_i,a_{i+1})=
\begin{cases}
    \{2n+1:n\in\mathbb{N}\}\cap [a_i,a_{i+1})&\text{if $i$ is even and }f_\alpha(i/2)=1\\
    \{2n:n\in\mathbb{N}\}\cap [a_i,a_{i+1})&\text{if $i$ is odd, or $i$ is even but }f_\alpha(i/2)=0
\end{cases}\]
By definition, each $B_{\alpha}$ has density $1/2$ and $B_{\alpha}\subseteq A$. Moreover, for $f_\alpha\neq f_\beta$, the intersection $B_{\alpha}\cap B_{\beta}$ will be empty on infinitely many intervals $[a_i,a_{i+1})$. Since the size of these intervals grow exponentially, we have that $d^-(B_{\alpha}\cap B_{\beta})=0$. Therefore, the set $\{(\emptyset, B_{\alpha}):\,\alpha\in I\}$ is {an antichain in $\mathbb{M}_{\geq 1/2}^-$ of size $2^{\aleph_0}$} , and it lies below the condition $(\emptyset,A)$ with $d^-(A)=1/2$.
\end{example}

~

Since the original proof is based on an incorrect corollary, we need to give a new proof that $\mathbb{M}_{\geq\varepsilon}^-$ is proper. Indeed, we prove it for the generalized forcing notion.

\begin{definition}
For each $m\in\mathbb{N}$, let $\varphi_m: \mathcal{P}(\mathbb{N})\rightarrow [0,\infty]$ be a lower semi-continuous submeasure on $\mathbb{N}$, such that $\lim\limits_{m\rightarrow\infty}\varphi_m(\{n\})=0$ for all $n\in\mathbb{N}$. Then we can define the submeasure $\varphi^-(A):=\liminf\limits_{m\rightarrow\infty}\varphi_m(A)$ and the corresponding Mathias forcing
\[\mathbb{M}^{\varphi^-}_{\geq\varepsilon}:=\{(s,A):\max(s)<\min(A)\text{ and }\varphi^-(A)\geq\varepsilon\}\]
where the ordering is defined as usual.
\end{definition}

It's clear that, when we take $\varphi_m(A)=\dfrac{|A\cap m|}{m}$, the above definition gives us $\mathbb{M}^-_{\geq\varepsilon}$, so it indeed generalizes Mathias forcing with lower density. Other examples satisfying the conditions include weighted density $\varphi_m(A)=\dfrac{\sum\limits_{i\in A\cap m} f(i)}{\sum\limits_{i< m} f(i)}$, {where $f$ is the weight function} (see also Erd\"os-Ulam submeasure, for instance in \cite{ilijasbook}), and the filter of cofinite sets (see Example \ref{cofinite} below).

\begin{fact}
\label{star_equal}
If $A=^*B$, then $\varphi^-(A)=\varphi^-(B)$.
\end{fact}
\begin{proof}
Suppose $A\cup A_0=B\cup B_0$ where $A_0, B_0$ are both finite. Then by subadditivity we have $\lim\limits_{m\rightarrow\infty}\varphi_m(A_0)=\lim\limits_{m\rightarrow\infty}\varphi_m(B_0)=0$. Hence \[\varphi^-(A)=\liminf\limits_{m\rightarrow\infty}\varphi_m(A)\leq \liminf\limits_{m\rightarrow\infty}\varphi_m(B\cup B_0)\leq \liminf\limits_{m\rightarrow\infty}\varphi_m(B)+\liminf\limits_{m\rightarrow\infty}\varphi_m(B_0)=\varphi^-(B)\]
and vice versa.
\end{proof}

\begin{theorem}
$\mathbb{M}_{\geq\varepsilon}^{\varphi^-}$ is proper for any $\varepsilon\geq 0$.
\end{theorem}

\begin{proof}
Take $\kappa$ large enough and $\mathcal{M}\prec H_\kappa$ countable with $\mathbb{M}_{\geq\varepsilon}^{\varphi^-}\in\mathcal{M}$. List all ordinal names $\dot{\alpha}_0,\dot{\alpha}_1,\dots$ in $\mathcal{M}$. Fix $(s,A_0)\in \mathbb{M}_{\geq\varepsilon}^{\varphi^-}\cap\mathcal{M}$ and we find an extension $(s,B)\leq (s,A_0)$ and {sets $\Gamma_n$ of ordinals in $\mathcal{M}$} such that $(s,B)\Vdash \dot{\alpha}_n\in \Gamma_n$ for all $n\in\mathbb{N}$.

{The construction is done by induction. In each stage $i+1$, we construct finite sets $B_{i+1}$ and $\Gamma_{n,i+1}$, and in the end we let $B=\bigcup_{i\geq 1} B_i$ and $\Gamma_n=\bigcup_{i\geq 1} \Gamma_{n,i}$. We also construct auxiliary sequences $\{A_{i}\}, \{N_{i}\},\{M_{i}\}$, where:
\begin{itemize}
    \item $A_0\supseteq A_1\supseteq A_2\supseteq\dots$ is a shrinking sequence of infinite subsets of $\mathbb{N}$;
    \item $0=N_0<N_1<N_2<\dots$ are the boundaries at each stage such that $A_i\cap N_{i}=A_{i+1}\cap N_i$;
    \item $0=M_0<M_1<M_2<\dots$ are witnesses of the limit $\varphi^-(A)=\liminf\limits_{m\rightarrow\infty}\varphi_m(A_{i+1})$ at each stage. Namely, $|\varphi_m(A_{i+1})-\varepsilon|$ is small enough for all $m\geq M_{i+1}$.
\end{itemize} Each stage $i+1$ will be a recursive construction on its own, where we construct the sequence $A_{i}=A_{i+1,0}\supseteq A_{i+1,1}\supseteq \dots\supseteq A_{i+1,l}=A_{i+1}$ such that $A_{i+1,k}\cap N_{i}=A_{i}\cap N_{i}$ for all $k$, and therefore $A_{i}\cap N_{i}=A_{i+1}\cap N_{i}$.}
\begin{itemize}
    \item Stage $0$:\\
    {$A_0$ is given by the condition $(s,A_0)$ we fix. Start with $M_0=N_0=0$.}
    \item Stage $i+1$:\\
    There are finitely many possible ways to extend $s$ with elements in $\bigcup_{1\leq r\leq i}B_r$ defined in previous stages, since each one of the set $B_r$ is finite. Enumerate all such extensions as $s_i^0,s_i^1,\dots,s_i^p$. {Start with $A_{i+1,0}=A_i$. Let $\{(n_k,j_k):\,0\leq k<(p+1)(i+2)\}$ enumerate all pairs $(n,j)$ such that $0\leq n\leq i+1$ and $0\leq j\leq p$. In each step $k$, ask (in $\mathcal{M}$) whether:}
    \[\exists C\subseteq (A_{i+1,k}\setminus N_i)\ \exists \beta\in\text{Ord}\ (s_i^{j_k}, C)\Vdash \dot{\alpha}_{n_k}=\check{\beta}\]
    \begin{enumerate}
        \item If the answer is ``no'', let $A_{i+1,k+1}=A_{i+1,k}$.
        \item If the answer is ``yes'', {choose a pair of witnesses $C_{i+1,k}$ and $\beta_{i+1,k}$ for $C$ and $\beta$, and let $A_{i+1,k+1}=(A_i\cap N_i)\cup C_{i+1,k}$. Note that $(s_i^{j_k}, C_{i+1,k})$ is a condition and therefore $\varphi^-(A_{i+1,k+1})\geq \varepsilon$.}
    \end{enumerate}
    In the end we let $A_{i+1}= A_{i+1,(p+1)(i+2)-1}$ and {$\Gamma_{n,i+1}=\{\beta_{i+1,k}:\,n_k=n\}$} for each $n$. Clearly $A_{i+1}\cap N_i=A_i\cap N_i$ and {each $\Gamma_{n,i+1}$ is a finite set of ordinals in $\mathcal{M}$}.\vspace{1.5mm}\\
    For the boundary at this stage, since $\varphi^{-}(A_{i+1})\geq \varepsilon$, we can find $M_{i+1}>M_i$ such that $\varphi_m(A_{i+1})\geq \varepsilon-\dfrac{1}{2^{i+1}}$ for all $m\geq M_{i+1}$. By lower semi-continuity, for each $m$, {there is $L_{i+1,m}$} such that $|\varphi_m(A_{i+1})-\varphi_m(A_{i+1}\cap n)|\leq \dfrac{1}{2^{i+1}}$ for all $n\geq L_{i+1,m}$. Let $N_{i+1}=N_i+\max\{L_{i+1,m}:\, m< M_{i+1}\}$.\vspace{1.5mm}\\
    For the last action at this stage, let $B_{i+1}=A_i\cap [N_i, N_{i+1})$. Note that $B_{i+1}$ is not a truncate of $A_{i+1}$, but rather, a truncate of $A_{i}$ from the previous stage.
\end{itemize}
Finally, let $B=\bigcup_{i\geq 1} B_i$ and $\Gamma_n=\bigcup_{i\geq 1} \Gamma_{n,i}$ as we promised. {Notice that all $B_{i+1}$ are disjoint. Also, $A_{i}\cap N_{i}=A_{i+1}\cap N_{i}$ for all $i$, hence we have $B\cap N_{i+1}=A_i\cap N_{i+1}$.} Now we prove that $B$ and $\Gamma_n$ we constructed do satisfy the requirements.
\begin{itemize}
    \item $B\subseteq A$.\\
    Each $B_{i+1}$ is a truncate of $A_i$, and $A_i\subseteq\dots\subseteq A_0=A$. {In fact, $B\subseteq A_i$ for all $i$.}
    \item $\varphi_m(B\cap n)\geq \varepsilon-\dfrac{1}{2^{i-1}}$ holds if for some $i\geq 1$ we have $M_i\leq m<M_{i+1}$ and $n\geq N_{i+1}$, and hence $\varphi^-(B)\geq \varepsilon$.

    Since $m\geq M_i$, we have $\varphi_m(A_i)\geq \varepsilon-\dfrac{1}{2^i}$. Since $m<M_{i+1}$ and therefore $N_{i+1}>L_{i+1,m}$, we have $|\varphi_m(A_i)-\varphi_m(A_i\cap N_{i+1})|\leq \dfrac{1}{2^i}$. Thus $\varphi_m(B\cap n)\geq \varphi_m(B\cap N_{i+1})=\varphi_m(A_i\cap N_{i+1})\geq \varphi_m(A_i)-|\varphi_m(A_i)-\varphi_m(A_i\cap N_{i+1})|\geq \varepsilon-\dfrac{1}{2^{i-1}}$.
    \item $(s,B)\Vdash \dot{\alpha}_n\in\Gamma_n$ for all $n$.\\
    Take an arbitrary extension $(t,D)\leq (s,B)$ such that $(t,D)\Vdash \dot{\alpha}_n=\gamma$ for some $\gamma\in$ Ord and some $n\in\mathbb{N}$. Then $t$ extends $s$ with elements from $\bigcup_{1\leq r\leq i}B_r$ for some $i$. Without loss of generality suppose $i\geq n$. So by our construction, $t=s_i^j$ for some $j$. In the corresponding step $k$ of stage $i+1$, we asked {in $\mathcal{M}$} whether
    \[\exists C\subseteq (A_{i+1,k}\setminus N_i)\ \exists \beta\in\text{Ord}\ (s_i^j, C)\Vdash \dot{\alpha}_n=\check{\beta}\]
    {By elementarity of $\mathcal{M}\prec H_\kappa$, the answer we get when asking in $\mathcal{M}$ is the same as asking in $H_\kappa$}, and the answer must be ``yes'', because $(s_i^j,D)\Vdash \dot{\alpha}_n=\gamma$ and $D\subseteq B\subseteq A_{i+1,k}$, and thus $(s_i^j,D\cap (A_{i+1,k}\setminus N_i))\Vdash \dot{\alpha}_n=\gamma$. Since the answer is ``yes'', we must've set $A_{i+1,k+1}=(A_i\cap N_i)\cup C_{i+1,k}$ at this step and thus $D\subseteq B\subseteq A_{i+1}\subseteq^* C_{i+1,k}$. Therefore $(s_i^j, D\cap C_{i+1,k})$ is a common extension of $(s_i^j, C_{i+1,k})$ and $(s_i^j, D)$. Now $(s_i^j, C_{i+1,k})\Vdash \dot{\alpha}_n=\beta_{i+1,k}$. On the other hand, $(s_i^j,D)\Vdash \dot{\alpha}_n=\gamma$. Hence $\beta_{i+1,k}=\gamma$. That is to say, $\gamma\in \Gamma_{n,i+1}$, and hence $(t, D)\Vdash \dot{\alpha}_n\in\Gamma_n$. By the arbitrariness of $(t,D)$, we see that $(s, B)\Vdash \dot{\alpha}_n\in\Gamma_n$.
\end{itemize}
\end{proof}

\begin{corollary}
$\mathbb{M}_{\geq\varepsilon}^{-}$ is proper for any $0\leq \varepsilon\leq 1$.\qed
\end{corollary}

~

Now the remaining question is whether $\mathbb{M}^{\varphi^-}_{\geq \varepsilon}$ adds a dominating real. The answer is: it depends. We give one example of such forcing not adding dominating real{s} and one that does.

\begin{example}
\label{cofinite}
Consider $\mathbb{M}^{\varphi^-}_{\geq \varepsilon}$ where
\[\varphi_m(A)=\begin{cases}1 &\quad m\in A\\ 0 &\quad m\notin A\end{cases}\]
and $0<\varepsilon\leq 1$. Then $\varphi^-(A)\geq \varepsilon$ iff $A$ is cofinite. Hence Cohen forcing densely embeds in $\mathbb{M}^{\varphi^-}_{\geq \varepsilon}$ by the map $s\mapsto ({s^{-1}(1)},\omega\setminus\dom(s))$. Therefore it does not add dominating real{s}.
\end{example}

{Next, we will show that $\mathbb{M}_{\geq\varepsilon}^{-}$ adds a dominating real. The proof relies on the following (weak) Darboux property of lower density:}

\begin{lemma}
{For any $A\subseteq \mathbb{N}$ with $d^-(A)=\delta>\varepsilon$, there is $A'\subset A$ such that $d^-(A)=\varepsilon$.}
\end{lemma}
\begin{proof}
{See \cite{MR3597402} Corollary 1\&2.}
\end{proof}

\begin{theorem}
$\mathbb{M}_{\geq\varepsilon}^{-}$ adds a dominating real for any $0<\varepsilon\leq 1$.
\end{theorem}

\begin{proof}
Let $G$ be the generic filter and $g=\bigcup \{s:\, \exists A\,(s,A)\in G\}$. By {the} Darboux property, we can fix a maximal antichain $M\subseteq \mathbb{M}_{\geq\varepsilon}^-$ such that $d^-(A)=\varepsilon$ for all $(s,A)\in M$. Suppose $M\cap G=\{(s_0,A_0)\}$. Enumerate $A_0$ as $a_0<a_1<a_2<\dots$ and let $A_n=\{a_i:\, i=2^{n}k,\ k\in\mathbb{N}^+\}$ {for $n\geq 1$}. Clearly $A_0\supset A_1\supset A_2\supset \dots$ and $d^+(A_n)=\dfrac{1}{2^{n}}d^+(A_0)\leq \dfrac{1}{2^{n}}$. Also, $d^-(A_0\setminus A_n)=\dfrac{2^{n}-1}{2^{n}}\varepsilon$. Therefore for all $A\subseteq A_0$ with $d^-(A)=\varepsilon$ and for all $n\geq 0$, we have $$d^-(A\cap A_n)\geq d^-(A)-d^-(A\setminus A_n)\geq \varepsilon - d^-(A_0\setminus A_n) =\dfrac{\varepsilon}{2^{n}}$$ and thus $A\cap A_n$ is nonempty.

In $V[G]$, define function $F:\omega\rightarrow\omega$ by
\[{F}(n)=\min\{k:\,k\in g\cap A_n \}\]
Such $k$ exists because $D_{n}=\{(s,A): \exists k\in s\cap A_n\}$ is dense below $(s_0,A_0)$, since $A\cap A_n$ is nonempty for all $(s,A)\leq (s_0,A_0)$. We show that $F$ is a dominating real.

Fix $f\in (\omega^\omega)^V$ and $(s,A)\in\mathbb{M}_{\geq\varepsilon}^-$ such that $(s,A)\leq (s_0,A_0)$. Let \[A_f=\bigcup_{n\in\omega}\{k\in A_n: k\leq f(n)\}\] {We claim that $d^-(A\setminus A_f)= \varepsilon$.} Indeed, since each summand of $A_f$ is finite, for all $m\in\mathbb{N}$ we have
\begin{align*}
    d^+(A_f)&=d^+(\bigcup_{n=1}^\infty\{k\in A_n: k\leq f(n)\})\\
    &=d^+(\bigcup_{n= m}^\infty\{k\in A_n: k\leq f(n)\})\\
    &\leq d^+(\bigcup_{n= m}^\infty A_n)\\
    &=d^+(A_m)\leq \dfrac{1}{2^{m}}
\end{align*}
and therefore $d(A_f)=0$. Thus $(s,A\setminus A_f)$ is a condition. Then, {since the infinite part $A\setminus A_f$ has excluded any $k\in A_n$ with $k\leq f(n)$, we have} \[{(s, A\setminus A_f)\Vdash \forall n\ (\dot F(n)\geq\min(A\setminus A_f)\rightarrow f(n)<\dot F(n)).}\]
{Note that $F(n)\in A_n$ and $\min(A_n)\geq 2^n$, and therefore this translates to}
\[{(s, A\setminus A_f)\Vdash \forall n>\log_2 \min(A\setminus A_f)\ (f(n)<\dot F(n)).}\]
Hence $D_f=\{(s,A):\, (s,A)\Vdash\exists m\forall n>m\ (f(n)<\dot F(n))\}$ is dense below $(s_0,A_0)$ for all ground model $f\in \omega^\omega$.
\end{proof}

{Note that for the submeasure $\varphi^-$ in Example \ref{cofinite}, if we attempt to apply the same proof, $A_f$ will be infinite and therefore $(s,A\setminus A_f)$ is no longer a condition.}

\vspace{8mm}

We switch to another density Mathias forcing. Instead of $d^-(A)\geq \varepsilon$, the condition is now $d^-(A)>0$. This forcing notion was also investigated by Matthew Harrison-Trainor, Liu Lu and Patrick Lutz from a recursion-theoretic perspective in \cite{harrisontrainor2023coding}. The following proof is inspired by Proposition 4.5 of their paper, {where they constructed $A_X\subseteq \mathbb{N}$ for a fixed set $X\subseteq \mathbb{N}$, such that any subset $B\subseteq A_X$ with $d^-(B)>0$ computes $X$ uniformly.}

\begin{theorem}
$\mathbb{M}^-_{>0}$ collapses {the} continuum to $\omega$.
\end{theorem}
\begin{proof}
Fix a bijection $f:\mathbb{N}\rightarrow\mathbb{N}\times\mathbb{N}\times\mathbb{N}$. We denote $f(n)$ by $\lr{a_n,b_n,c_n}$. In the following we construct disjoint intervals $I_{a_n,b_n,c_n}$ for each $n$ inductively. Later, for some fixed condition $(s,A)$ and set $S\in (2^\omega)^V$, we will code {$S$} with a one-step extension $(s,A_0)\leq (s,A)$. {The $b$-th digit of $S$ will be stored infinitely many times in $A_0\cap I_{a,b,c}$ for the smallest $a$ such that $d^-(A)>{1}/{(a+2)}$ and for all $c\in\mathbb{N}$.}

Let $I_{a_0,b_0,c_0}=[0,1)$. Suppose $I_{a_n,b_n,c_n}=[i_n,i_{n+1})$ is constructed. Then we let \[I_{a_{n+1},b_{n+1},c_{n+1}}=[i_{n+1},(2a_{n+1}+4)^{2c_{n+1}+3}(2a_{n+2}+4)\,i_{n+1}).\] We further divide each $I_{a_n,b_n,c_n}$ into intervals \[I_{a_n,b_n,c_n}^k=[(2a_n+4)^k\cdot i_{n+1},(2a_n+4)^{k+1}\cdot i_{n+1})\] for $0\leq k\leq 2c_n+2$ and \[I_{a_n,b_n,c_n}^{2c_n+3}=[(2a_n+4)^{2c_n+3}\cdot i_{n+1},(2a_{n+1}+4)^{2c_{n+1}+3}(2a_{n+2}+4)\,i_{n+1}).\]
Later in the coding step $A_0\subseteq A$, depending on whether the $b$-th digit of $S$ is $1$ or $0$, we will either store the information in the corresponding odd intervals $I_{a,b,c}^1,I_{a,b,c}^3,\dots,I_{a,b,c}^{2c+1}$, or in the even intervals $I_{a,b,c}^2,I_{a,b,c}^4,\dots,I_{a,b,c}^{2c+2}$. {The first interval $I_{a,b,c}^{0}$ and last interval $I_{a,b,c}^{2c+3}$} will act as buffer areas; they aren't necessary but having such extra intervals makes the proof easier.

\begin{claim}
Fix $a'\in\mathbb{N}$. For any $A$ such that $d^-(A)>1/(a'+2)$, and any $A_0\subseteq A$ satisfying
\begin{itemize}
    \item For $a=a'$ and $b,c\in \mathbb{N}$, either \[I_{a,b,c}\cap (A\setminus A_0)\subseteq I_{a,b,c}^2 \cup I_{a,b,c}^4\cup \dots \cup I_{a,b,c}^{2c+2}\] or \[I_{a,b,c}\cap (A\setminus A_0)\subseteq I_{a,b,c}^1 \cup I_{a,b,c}^3\cup \dots \cup I_{a,b,c}^{2c+1};\]
    \item For $a\neq a'$ and $b,c\in \mathbb{N}$, $I_{a,b,c}\cap (A\setminus A_0)= \emptyset$.
\end{itemize}
we have $d^-(A_0)>0$.
\end{claim}
\begin{proof}[Proof of claim]
Since $d^-(A)>1/(a'+2)$, there is $n$ such that $\dfrac{|x\cap A|}{x}>\dfrac{1}{a'+2}$ for all $x>n$. Consider any interval $I_{a,b,c}^k=[i_r,i_{r+1})$ such that $i_r\gg n$. {It is of one of the following three cases:} 
\begin{enumerate}
    \item $a=a'$, $k\neq 0$ and $I_{a',b,c}^{k-1}=\big[\dfrac{i}{2a'+4},i\big)$ has empty intersection with $A\setminus A_0$:\\
    By counting we have $|I_{a',b,c}^{k-1}\cap A_0|=|I_{a',b,c}^{k-1}\cap A|\geq |i\cap A|-\dfrac{i}{2a'+4}>\dfrac{i}{2a'+4}$. Thus $\dfrac{|x\cap A_0|}{x}\geq \dfrac{|I_{a',b,c}^{k-1}\cap A_0|}{x}> \dfrac{1}{(2a'+4)^2}$ for all $x\in I_{a',b,c}^k$.
    \item $a=a'$, $k\neq 0$ and $I_{a',b,c}^{k-1}\cap (A\setminus A_0)\neq \emptyset$:\\
    In this case we must have $k\geq 2$, since $I_{a',b,c}^0\cap (A\setminus A_0)=\emptyset$. Moreover, $I_{a',b,c}^{k-2}$ and $I_{a',b,c}^{k}$ both have empty intersection with $A\setminus A_0$. Therefore $\dfrac{|x\cap A_0|}{x}\geq \dfrac{|I_{a',b,c}^{k-2}\cap A_0|}{x}> \dfrac{1}{(2a'+4)^3}$ for all $x\in I_{a',b,c}^k$.
    \item $a= a'$ and $k=0$, or $a\neq a'$;\\
    Notice that $I_{a,b,c}\cap (A\setminus A_0)$ is empty when $a\neq a'$. Also we have a buffer area $I_{a,b,c}^{2c+3}$ at the end of each $I_{a,b,c}$ which always has empty intersection with $A\setminus A_0$. Hence in case 3 we always have $\big[\dfrac{i}{2a'+4},i\big)\cap (A\setminus A_0)=\emptyset$. By the same argument as in case 1, we have $\dfrac{|x\cap A_0|}{x}> \dfrac{1}{(2a'+4)^2}$ for all $x\in I_{a,b,c}^k$.
\end{enumerate}
Hence, we have $d^-(A_0)\geq \dfrac{1}{(2a'+4)^3}>0$.\end{proof}

~

Now let's define the coding. Let $G$ be the generic filter and $g=\bigcup\{s:\,(s,A)\in G\}$. Define $F: \mathbb{N}\times \mathbb{N}\rightarrow 2^\omega$ in $V[G]$ as follows. Fix $a,b,n\in\mathbb{N}$. Let
\[{l(a,b)=\min\{l\in\mathbb{N}:\, \exists c\exists k \,(l\in g \cap I_{a,b,c}^k\text{ and }1\leq k\leq 2c+2)\}}.\]
To see such $l$ exists, notice that missing all the consecutive intervals $I_{a,b,c}^k$ for $1\leq k\leq 2c+2$ forces $\dfrac{|x\cap A|}{x}$ to temporarily drop below $\dfrac{1}{(2a+4)^{2c+2}}$, and therefore any set $A$ with $d^-(A)>0$ cannot miss all these intervals for fixed $a,b$. Now define $F(a)\in 2^\omega$ by
\[{F(a)(b)=\begin{cases}
    1 & \qquad \text{if $l(a,b)\in I_{a,b,c}^k$ where $k$ is odd}\\
    0 & \qquad \text{if $l(a,b)\in I_{a,b,c}^k$ where $k$ is even}
\end{cases}}\]
and $\dot{F}$ be a name for $F$. It remains to show $D_S=\{(s,A):\, \exists a\, (s,A)\Vdash \dot{F}(a)=S\}$ is dense for all $S\in (2^\omega)^V$.

Fix $(s,A)$ and $S\in (2^\omega)^V$. Choose the smallest natural number $a'$ such that $d^-(A)>1/(a'+2)$. Now, for all $b,c\in\mathbb{N}$, we remove from $A$ all the even intervals $I_{a',b,c}^k$ if $S(b)=1$, and all the odd intervals $I_{a',b,c}^k$ if $S(b)=0$. That is,
\[A_0=A\setminus\big(\bigcup\Big\{I_{a',b,c}^k:\,S(b)\equiv k+1\text{ mod } 2,\,1\leq k\leq 2c+2\Big\}\big)\]
By the above claim, $(s,A_0)$ is a condition. Now, $(s,A_0)$ already decides whether $l(a',b)$ lies in an odd interval (or even interval) $I_{a',b,c}^k$, because either all the even intervals (or all the odd intervals) have empty intersection with $A_0$, depending on the value of $S(b)$. Thus $(s,A_0)\Vdash \dot{F}(a')=S$. Therefore, $D_S$ is dense for all ground model $S\in 2^\omega$.
\end{proof}

\section{Silver Forcings with Density}

In this section we investigate the Silver forcings that are yet to be categorized in \cite{LMS-R}. The first density Silver forcing we investigate is $\mathbb{V}^-_{\geq \varepsilon}$. Unlike the Mathias case, this forcing collapses {the} continuum. The intuition is, each condition $f$ is allowed to contain infinite positive information, and hence we can code any real in the ground model via one-step extension, by simply adding it to the domain of $f$. The tough part is how to recover this information from the generic. Namely, we need to find which bits of $g$ were used to code $f$ back in the one-step extension. We start with a lemma.

\begin{lemma}
\label{interval}
Fix $A$ such that $d^-(A)=\varepsilon$. Then there exists an infinite sequence $N_0<N_1<N_2<\dots$ such that for every $B\subseteq A$ with $d^-(B)=\varepsilon$, we always have
\[\lim_{i\rightarrow\infty}\frac{|(A\setminus B)\cap [N_i,N_{i+1})|}{N_{i+1}-N_i}=0\]
Moreover, {$\{N_i\}_{i\in\omega}$ can grow arbitrarily fast}.
\end{lemma}

\begin{proof}
Take $\{N_i\}_{i\in\omega}$ such that $\dfrac{|A\cap N_i|}{N_i}\leq \varepsilon+\dfrac{1}{i}$ and $N_{i+1}>2N_i$ for all $i$. Such sequence exists since $d^-(A)=\varepsilon$. Suppose we have some $B\subseteq A$ with $d^-(B)=\varepsilon$ such that
\[\frac{|(A\setminus B)\cap [N_i,N_{i+1})|}{N_{i+1}-N_i}\geq \delta\]
for infinitely many $i\in\mathbb{N}$ and some fixed $\delta>0$. Then for these $i\in\mathbb{N}$, we have
\[\frac{|B\cap N_{i+1}|}{N_{i+1}}\leq\frac{|A\cap N_{i+1}|-|(A\setminus B)\cap [N_i,N_{i+1})|}{N_{i+1}}\leq \varepsilon+\frac{1}{i+1}-\frac{\delta}{2}\]
and therefore $d^-(B)\leq \varepsilon-\dfrac{\delta}{2}$. Contradiction. The moreover part of the lemma follows from the proof itself since we never imposed a cap on the growth of $N_{i+1}$.
\end{proof}

\begin{theorem}
For any $0<\varepsilon\leq 1$, forcing with $\mathbb{V}^-_{\geq\varepsilon}$ collapses {the} continuum to $\omega$.
\end{theorem}

\begin{proof}
Let $G$ be the generic filter and $g=\bigcup G$ is a total function $g:\omega\rightarrow 2$. Fix a maximal antichain $M\subseteq \mathbb{V}_{\geq\varepsilon}^-$ such that $d^-(\omega\setminus\dom(f))= \varepsilon$ for all $f\in M$. Suppose $M\cap G=\{f_0\}$. Now that $d^-(\omega\setminus\dom(f_0))= \varepsilon$, we take the sequence $\{N_i\}_{i\in\omega}$ satisfying Lemma \ref{interval}. Also fix a bijection $\langle\cdot,\,\cdot\rangle:\mathbb{N}\times\mathbb{N}\rightarrow\mathbb{N}$. Let $I_i^j$ be the interval $[N_{\lr{i,j}},N_{\lr{i,j}+1})$.

We'd like to design a coding $F$ with the help of $g$, such that $\{F(n)\}_{n\in\omega}$ enumerates $(2^\omega)^V$. For each interval $I_i^j$, we list the first $\lfloor\dfrac{\varepsilon\cdot|I_i^j|}{2^{i+1}}\rfloor$ elements of $I_i^j\setminus \dom(f_0)$, and check the value of $g$ at these slots. We say that $I_i^j$ is positive, if there are more $1$'s {in $g$} than $0$'s among the slots we check; otherwise it's negative. {Notice that, since $d^-(\omega\setminus \dom (f_0))=\varepsilon$, there will be enough elements in $I_i^j\setminus \dom(f_0)$ for $k=\lr{i,j}$ large enough.}

Fix $n\in\mathbb{N}$. {We construct $F(n)$ together with a sequence $\{m_{i,n}\}_{i\in\omega}$}. The idea is to use an interval $I_i^{m_{i,n}}$ to code $F(n)(i)$, then several intervals afterwards to code the position of the next interval $I_{i+1}^{m_{i+1,n}}$. Start with $m_{0,n}=n$. Suppose $m_{i,n}$ is constructed. We let $F(n)(i)=1$ if $I_i^{m_{i,n}}$ is positive, and $F(n)(i)=0$ if $I_i^{m_{i,n}}$ is negative. Then let $m_{i+1,n}$ be the largest number $j\geq 0$ such that $I_i^{m_{i,n}+1},\dots,I_i^{m_{i,n}+j}$ are all positive ($j=0$ if the first interval $I_i^{m_{i,n}+1}$ is already negative). Such finite $j$ exists by the usual genericity argument: {for any $f\leq f_0$ and fixed $i$, by the above lemma, $\lim\limits_{j\rightarrow\infty}\dfrac{|(\dom(f)\setminus\dom(f_0))\cap I_i^j|}{|I_i^j|}=0$ and hence $f$ only decides whether $I_i^j$ is positive/negative for finitely many $j$'s. Then we can extend $f'\leq f$ to ensure at least one of the intervals $I_i^{m_{i,n}+j}$ is negative, and therefore $j$ is finite.}

The above definition gives us a map $F:\omega\rightarrow 2^\omega$ in $V[G]$. Let $\dot F$ be its name in $V$, and $\dot{m}_{i,n}$ be a name for each $m_{i,n}$. For every $A\in (2^\omega)^V$, define the set
$D_A=\{f\in\mathbb{V}^-_{\geq\varepsilon}: \exists n\ f\Vdash \dot F(n)=A\}$. It suffices to prove $D_A$ is dense below $f_0$.

{Fix $f\leq f_0$ and $A\in (2^\omega)^V$}. By Lemma \ref{interval}, for each $i\in\mathbb{N}$ we can find $n_i$ large enough such that
\[|(\dom(f)\setminus\dom(f_0))\cap I_i^j|< \dfrac{\varepsilon\cdot|I_i^j|}{2^{i+3}}\]
for all $j\geq n_i$. Now we take the extension $f'\leq f$ such that
{{\footnotesize 
\[f'(k)=\begin{cases}
    1 & k\notin \dom(f)\text{ and $k$ is among the first $\lfloor\dfrac{\varepsilon\cdot|I_i^{n_i}|}{2^{i+1}}\rfloor$ elements of $I_i^{n_i}\setminus \dom(f_0)$ for }A(i)=1 \\
    0 & k\notin \dom(f)\text{ and $k$ is among the first $\lfloor\dfrac{\varepsilon\cdot|I_i^{n_i}|}{2^{i+1}}\rfloor$ elements of $I_i^{n_i}\setminus \dom(f_0)$ for }A(i)=0 \\
    1 & k\notin \dom(f)\text{ and $k$ is among the first $\lfloor\dfrac{\varepsilon\cdot|I_i^{n_i+j}|}{2^{i+1}}\rfloor$ elements of $I_i^{n_i+j}\setminus \dom(f_0)$ for }1\leq j\leq n_{i+1}\\
    0 & k\notin \dom(f)\text{ and $k$ is among the first $\lfloor\dfrac{\varepsilon\cdot|I_i^{n_i+n_{i+1}+1}|}{2^{i+1}}\rfloor$ elements of $I_i^{n_i+n_{i+1}+1}\setminus \dom(f_0)$} 
\end{cases}\]
}and $f'(k)$ is not defined except for $k\in\dom(f)$ or when $k$ satisfies one of the above four conditions.} Notice that for each $i$, we extended the domain on the intervals $I_i^j$ for only finitely many $j$'s, and in each $I_i^j$ the extended part has density at most $\dfrac{\varepsilon}{2^{i+1}}$. Thus $d(\dom(f')\setminus\dom(f))=0$. So $f'$ is a condition.

Now we claim $f'\Vdash \dot F(n_0)=A$. {This is because $f'$ already has enough information to decide whether each relevant interval is positive/negative, since $(\dom(f')\setminus\dom(f))\cap I_i^j$ has occupied at least ${3}/{4}$ of the slots that we check when deciding whether $I_i^j$ is positive. Now by definition we have $m_{0,n_0}=n_0$, and the intervals $I_i^{n_i+1},\dots,I_i^{n_i+n_{i+1}+1}$ indicate $m_{i+1,n_0}=n_{i+1}$. Therefore by induction $f'\Vdash \dot{m}_{i,n_0}=n_i$ for all $i$. Finally, the intervals $I_i^{n_i}$ ensure $F(n_0)(i)=A(i)$. Therefore $f'\Vdash \dot{F}(n_0)(i)=A(i)$}. We conclude that $D_A$ is dense below $f_0$, and thus $F$ enumerates $(2^\omega)^V$.
\end{proof}

~

The last density forcing we investigate is $\mathbb{V}^+_{>0}$. Since $d^+(A)>0$ is equivalent to saying $A$ is not a density zero set, this forcing naturally splits into a two-step iteration. 
\begin{lemma}
$\mathbb{V}^+_{>0}$ is forcing equivalent with the two-step iteration $\mathcal{P}(\omega)/\mathcal{Z}_0*\mathbb{V}_{ground}(\dot{\mathcal{F}})$, where $\mathcal{Z}_0=\{A\subset\omega:\,d(A)=0\}$, $\dot{\mathcal{F}}$ is the name for $\mathcal{P}(\omega)/\mathcal{Z}_0$-generic filter, and $\mathbb{V}_{ground}(\dot{\mathcal{F}})$ consists of partial functions $f:\omega\rightarrow 2$ in the ground model such that $\omega\setminus\dom(f)\in \dot{\mathcal{F}}$.
\end{lemma}
\begin{proof}
Define the map
\begin{align*}
    j:\,\mathbb{V}^+_{>0} &\rightarrow \mathcal{P}(\omega)/\mathcal{Z}_0*\mathbb{V}_{ground}(\dot{\mathcal{F}})\\
    f &\mapsto \lr{[\omega\setminus\dom(f)],{\check{f}}}
\end{align*}
It's clear that $f\leq g$ implies $j(f)\leq j(g)$. Suppose $f\perp g$. Then either $f(n)\neq g(n)$ for some $n$, or $d(\omega\setminus(\dom(f)\cup\dom(g)))=0$. In both cases $j(f)\perp j(g)$. To show that $j$ is a dense embedding, fix $\lr{[A],{\dot{f}}}\in \mathcal{P}(\omega)/\mathcal{Z}_0*\mathbb{V}_{ground}(\dot{\mathcal{F}})$. So $[A]\Vdash \omega\setminus\dom(\dot{f})\in \dot{\mathcal{F}}$. Hence $A \subseteq \omega\setminus\dom(f)$ modulo a density zero set. Let $f'=f\cup \{(n,0):\,n\notin A\land n\notin \dom(f)\}$. Then $j(f')=\lr{[A],\dot{f}'}\leq \lr{[A],\dot{f}}$.
\end{proof}

This intermediate model $V[\mathcal{F}]$ will help us show that $\mathbb{V}^+_{>0}$ collapses {the} continuum.
\begin{theorem}
    Forcing with $\mathbb{V}^+_{>0}$ collapses {the} continuum to $\omega$.
\end{theorem}
\begin{proof}
Let $\mathcal{F}$ be the generic filter in $\mathcal{P}(\omega)/\mathcal{Z}_0$. In $V[\mathcal{F}]$, we inductively construct a partition $\mathbb{N}=\bigcup_{n\in \omega}A_n$ satisfying the following:
\begin{itemize}
    \item Each $A_n$ is in the ground model and {$\omega\setminus \bigcup\limits_{i=0}^n A_i\in \mathcal{F}$ for all $n$}.
    \item For each $F\in\mathcal{F}$, there is $n$ such that $A_n\cap F$ is infinite.
\end{itemize}

{Start with $A_0=\emptyset$. Clearly $\omega\setminus A_0\in\mathcal{F}$.} Suppose $A_i$ is constructed for $i\leq n$. Enumerate the elements of {$\omega\setminus \bigcup\limits_{i=0}^n A_i$} as $b_n^0<b_n^1<b_n^2<\dots$ and consider the disjoint union $\omega\setminus \bigcup\limits_{i=0}^n A_i=B_n^1\cup B_n^2$ where $B_{n}^1=\{b_n^i:\, i\text{ is odd}\}$ and $B_n^2=\{b_n^i:\, i\text{ is even}\}$. {Both $B_n^1$ and $B_n^2$ are in the ground model, since by induction hypothesis $A_i$ is in the ground model for $i\leq n$. Notice that $\mathcal{F}$ is an ultrafilter over the ground model, and by induction hypothesis $\omega\setminus \bigcup\limits_{i=0}^n A_i\in \mathcal{F}$; therefore exactly one of $B_n^1,B_n^2$ is in $\mathcal{F}$.} Let $A_{n+1}$ be the one \emph{not} in $\mathcal{F}$. Then $\omega\setminus \bigcup\limits_{i=0}^{n+1} A_i\in \mathcal{F}$.

It's clear that $\mathbb{N}=\bigcup_{n\in \omega}A_n$ and {$d(A_n)=\dfrac{1}{2^{n}}$ for $n\geq 1$}. Now suppose $A_n\cap F$ is finite for all $n$. Then \[d^+(F)=d^+(\bigcup_{i\in\omega} (A_i\cap F))=d^+(\bigcup_{i\geq n} (A_i\cap F))\leq d^+(\bigcup_{i\geq n} A_i)= \dfrac{1}{2^{n-1}}\] for all $n\geq 1$. Hence $d(F)=0$ and thus $F\notin\mathcal{F}$. Finally, let $a_n^0<a_n^1<a_n^2<\dots$ enumerate each $A_n$, and define $\iota_n(i)=a_n^i$. Each $\iota_n$ is also in the ground model.

In $V$, we construct a sufficiently surjective function $\sigma:\, 2^\omega\rightarrow 2^\omega$. That is, for every $x\in 2^\omega$ and for every partial function $f:\,\omega\rightarrow 2$ such that $\omega\setminus\dom(f)$ is infinite, there is $y\in 2^\omega$ such that $f\subset y$ and $\sigma(y)=x$. This is done by transfinite induction: let $\{\lr{x_\alpha,f_\alpha}:\,\alpha<2^{\aleph_0}\}$ enumerate all such pairs of $x$ and $f$. At each stage, we choose $y_\alpha\supset f_\alpha$ which is not in the domain of $\sigma$ yet and let $\sigma(y_\alpha)=x_\alpha$. This is doable because $|\omega\setminus\dom(f_\alpha)|=\aleph_0$, and at the current stage, $\sigma$ is only defined at $<2^{\aleph_0}$ many points. Then $\sigma$ is obtained by taking union at limit stages.

Let $G$ be the generic filter in $\mathbb{V}_{ground}(\mathcal{F})$ and $g=\bigcup G$ be the generic real. In $V[\mathcal{F}][G]$, define
\begin{align*}
    \eta:\omega &\rightarrow (2^\omega)^{V}\\
    n &\mapsto \sigma(g\circ \iota_n)
\end{align*}
To see that this is well-defined, \textit{i.e.} $g\circ \iota_n\in (2^\omega)^V$, notice that each $A_n$ is in the ground model and $\omega\setminus A_n\in \mathcal{F}$. Hence, given any $f\in\mathbb{V}_{ground}(\mathcal{F})$ we can extend it to $f'\leq f$ such that $A_n\subseteq \dom(f')$. Therefore $f'\Vdash g\circ \iota_n\in (2^\omega)^V$.

Now we claim that $\eta$ is surjective. Fix $x\in(2^\omega)^{V}$ and take an arbitrary $f\in \mathbb{V}_{ground}(\mathcal{F})$. Then $A_n\cap(\omega\setminus\dom(f))$ is infinite for some $n$, or equivalently, $\omega\setminus \dom(f\circ \iota_n)$ is infinite. So there is $y\supset f\circ \iota_n$ such that $\sigma(y)=x$. Now let $f'=f\cup (y\circ \iota_n^{-1})$. Then $\dom(f')=\dom(f)\cup A_n$ and therefore it is a condition, and it follows from the definition of $\eta$ that $f'\Vdash \dot{\eta}(n)=x$ .
\end{proof}

~

To conclude, we've completed the following table from \cite{LMS-R}. The shaded boxes are the results proven in this paper. Note that if a forcing notion collapses {the} continuum to $\omega$, then it fails to be proper and also adds both dominating and Cohen reals.

\begin{center}
\begin{tabular}{|c|c|c|c|c|c|c|c|c|}
\hline
~ & $\mathbb{M}^-_{>0}$ & $\mathbb{M}^+_{>0}$ & $\mathbb{M}^-_{\geq\epsilon}$ & $\mathbb{M}^+_{\geq\epsilon}$ & $\mathbb{V}^-_{>0}$ & $\mathbb{V}^+_{>0}$ & $\mathbb{V}^-_{\geq\epsilon}$ & $\mathbb{V}^+_{\geq\epsilon}$ \\
\hline
proper & \cellcolor{lightgray}\xmark & \cmark & \cellcolor{lightgray}\cmark & \cmark & \xmark & \cellcolor{lightgray}\xmark & \cellcolor{lightgray}\xmark & \cmark\\
\hline
dominating & \cellcolor{lightgray}\cmark & \cmark & \cellcolor{lightgray}\cmark & \cmark & \cmark & \cellcolor{lightgray}\cmark & \cellcolor{lightgray}\cmark & \xmark\\
\hline
Cohen & \cmark & \cmark & \cmark & \cmark & \cmark & \cellcolor{lightgray}\cmark & \cellcolor{lightgray}\cmark & \xmark\\
\hline
\end{tabular}
\end{center}

\begin{question}
    How about the corresponding forcings with subscript $>\varepsilon$ for $\varepsilon\neq 0$? Are they proper? If so, do they add dominating and Cohen reals?
\end{question}

~

\noindent\textbf{Acknowledgements.} Thanks to Dilip Raghavan for advice on almost everything in this paper. Thanks to Zhang Zhentao for early discussion on properties of upper and lower density. Thanks to Will Johnson for a counterexample on upper density which is not included in this paper; it saved the author from trying to prove a lemma that turns out to be incorrect.

\end{spacing}
\nocite{*}
\bibliography{references}

\begin{thebibliography}{10}

\bibitem{MR1350295}
Tomek Bartoszy\'{n}ski and Haim Judah.
\newblock {\em Set theory}.
\newblock A K Peters, Ltd., Wellesley, MA, 1995.
\newblock On the structure of the real line.

\bibitem{MR2768685}
Andreas Blass.
\newblock Combinatorial cardinal characteristics of the continuum.
\newblock In {\em Handbook of set theory. {V}ols. 1, 2, 3}, pages 395--489. Springer, Dordrecht, 2010.

\bibitem{2022arXiv220614936B}
J{\"o}rg {Brendle}, Osvaldo {Guzm{\'a}n}, Michael {Hru{\v{s}}{\'a}k}, and Dilip {Raghavan}.
\newblock {Combinatorial properties of MAD families}.
\newblock {\em arXiv e-prints}, pages 1--43 arXiv:2206.14936, June 2022.

\bibitem{MR3129732}
J\"{o}rg Brendle and Dilip Raghavan.
\newblock Bounding, splitting, and almost disjointness.
\newblock {\em Ann. Pure Appl. Logic}, 165(2):631--651, 2014.

\bibitem{MR1686797}
J\"{o}rg Brendle and Saharon Shelah.
\newblock Ultrafilters on {$\omega$}---their ideals and their cardinal characteristics.
\newblock {\em Trans. Amer. Math. Soc.}, 351(7):2643--2674, 1999.

\bibitem{MR0969054}
R.~Michael Canjar.
\newblock Mathias forcing which does not add dominating reals.
\newblock {\em Proc. Amer. Math. Soc.}, 104(4):1239--1248, 1988.

\bibitem{MR3990958}
David Chodounsk\'{y} and Osvaldo Guzm\'{a}n.
\newblock There are no {P}-points in {S}ilver extensions.
\newblock {\em Israel J. Math.}, 232(2):759--773, 2019.

\bibitem{ilijasbook}
I.~Farah.
\newblock Analytic quotients: theory of liftings for quotients over analytic ideals on the integers.
\newblock {\em Mem. Amer. Math. Soc.}, 148(702):xvi+177, 2000.

\bibitem{MR2254542}
Ilijas Farah.
\newblock Analytic {H}ausdorff gaps. {II}. {T}he density zero ideal.
\newblock {\em Israel J. Math.}, 154:235--246, 2006.

\bibitem{MR0297560}
Serge Grigorieff.
\newblock Combinatorics on ideals and forcing.
\newblock {\em Ann. Math. Logic}, 3(4):363--394, 1971.

\bibitem{harrisontrainor2023coding}
Matthew Harrison-Trainor, Lu~Liu, and Patrick Lutz.
\newblock Coding information into all infinite subsets of a dense set, 2023.
\newblock arXiv:2306.01226.

\bibitem{MR2319159}
Fernando Hern\'{a}ndez-Hern\'{a}ndez and Michael Hru\v{s}\'{a}k.
\newblock Cardinal invariants of analytic {$P$}-ideals.
\newblock {\em Canad. J. Math.}, 59(3):575--595, 2007.

\bibitem{MR3142391}
Michael Hru\v{s}\'{a}k and Hiroaki Minami.
\newblock Mathias-{P}rikry and {L}aver-{P}rikry type forcing.
\newblock {\em Ann. Pure Appl. Logic}, 165(3):880--894, 2014.

\bibitem{MR1068126}
Claude Laflamme.
\newblock Zapping small filters.
\newblock {\em Proc. Amer. Math. Soc.}, 114(2):535--544, 1992.

\bibitem{LMS-R}
G.~Laguzzi, H.~Mildenberger, and B.~Stuber-Rousselle.
\newblock Mathias and {S}ilver forcing parametrized by density.
\newblock {\em Arch. Math. Logic}, pages 00881--7, 26 pp., 2023.

\bibitem{MR3597402}
Paolo Leonetti and Salvatore Tringali.
\newblock Upper and lower densities have the strong {D}arboux property.
\newblock {\em J. Number Theory}, 174:445--455, 2017.

\bibitem{MR4099835}
Dilip Raghavan.
\newblock The density zero ideal and the splitting number.
\newblock {\em Ann. Pure Appl. Logic}, 171(7):102807, 15 pp., 2020.

\bibitem{MR3615051}
Dilip Raghavan and Saharon Shelah.
\newblock Two inequalities between cardinal invariants.
\newblock {\em Fund. Math.}, 237(2):187--200, 2017.

\bibitem{MR1708146}
S.~Solecki.
\newblock Analytic ideals and their applications.
\newblock {\em Ann. Pure Appl. Logic}, 99(1-3):51--72, 1999.

\end{thebibliography}
\bibliographystyle{plain}

\end{document}